\documentclass{ifacconf}

\usepackage{graphicx}      
\usepackage{natbib}
\usepackage{amsmath,amssymb,amsopn,cases,mathrsfs,stmaryrd}
\usepackage{hhline}
\usepackage{blindtext,xcolor}
\usepackage{multirow}
\usepackage[bb=boondox]{mathalfa}

\renewcommand{\S}{\mathbb{S}}
\newcommand{\R}{\mathbb{R}}
\newcommand{\A}{\mathcal{A}}
\newcommand{\B}{\mathcal{B}}
\newcommand{\C}{\mathcal{C}}
\newcommand{\D}{\mathcal{D}}
\renewcommand{\L}{\mathscr{L}}
\newcommand{\boundary}{\mathscr{B}}
\newcommand{\domain}{\mathscr{D}}
\newcommand{\operator}{\mathscr{A}}
\newcommand{\N}{\mathbb{N}}

\DeclareMathOperator*{\He}{He}
\DeclareMathOperator*{\diag}{diag}
\DeclareMathOperator*{\col}{col}

\newtheorem{de}{Definition}
\newtheorem{theo}{Theorem}

\newtheorem{lemma}{Lemma}
\newtheorem{example}{Example}
\newenvironment{proof}[1][]{\textbf{Proof #1:~}}{\hfill$\square$\\}
\newenvironment{remark}[1][]{\begin{rem}}{\hfill$\blacksquare$\end{rem}}

\begin{document}
\begin{frontmatter}
\title{Integral Quadratic Constraints on Linear Infinite-dimensional Systems for Robust Stability Analysis}


\author[First]{Matthieu Barreau}
\author[Second]{Carsten W. Scherer}
\author[First]{Fr{\'e}d{\'e}ric Gouaisbaut}
\author[First]{Alexandre Seuret}

\address[First]{LAAS-CNRS, Universit\'e de Toulouse, CNRS, UPS, Toulouse, France. (e-mail: barreau,seuret,fgouaisb@laas.fr).}%
\address[Second]{Department of Mathematics, University of Stuttgart, Stuttgart, Germany. (e-mail: carsten.scherer@mathematik.uni-stuttgart.de).}

\thanks[]{This work is supported by the ANR project SCIDiS contract number 15-CE23-0014.}

\begin{abstract}
This paper proposes a framework to assess the
stability of an ordinary differential equation
which is coupled to a 1D-partial differential equation (PDE).
The stability theorem is based on a new result on
Integral Quadratic Constraints (IQCs) and expressed
in terms of two linear matrix inequalities with a moderate
computational burden. The IQCs are not generated using dissipation inequalities involving the whole state of an infinite-dimensional system, but by using
projection coefficients of the infinite-dimensional state. This permits
to generalize our robustness result to many other PDEs.
The proposed methodology is applied to a time-delay system and
numerical results comparable to those in the literature are obtained.
\end{abstract}

\end{frontmatter}

\section{Introduction}

Many processes are modeled by a finite-dimensional linear system subject to an infinite-dimensional uncertainty \citep{drillingArticle,norm}. Robustness analysis assesses the stability of the resulting coupled system. This paper focuses on an infinite-dimensional type of uncertainties modeled by a linear Partial Differential Equation. Attention has been paid on these special couplings, partly because Time-Delay Systems (TDS) are a subclass thereof \citep{norm,opac-b1100602,safi2017tractable}.

In recent decades, many articles have been dealing with asymptotic stability of TDS, using different robust control tools. A first method is based on the Lyapunov-Krasovskii theorem and requires to find a differentiable functional with suitable properties \citep{opac-b1100602}. Another approach considers dissipativity \citep{willems1972dissipative}. More recently, a general framework for dealing with coupled Ordinary Differential Equation (ODE)  and uncertainties has been proposed by \cite{fu1998robust,jun2001multiplier,megretski1997system}.
This requires to find IQCs to assess stability of the interconnection.

The methodology presented in \cite{seuret:hal-01065142} shows how to construct a Lyapunov functional for TDS based on the projection coefficients of the infinite-dimensional state. The resulting stability test is made up of tractable Linear Matrix Inequalities (LMIs) with a moderate computational burden. This approach has been applied to a transport equation in \cite{safi2017tractable}, to a heat equation in \cite{BAUDOUIN2019195} and to a wave equation in \cite{barreau2018lyapunov}. 

The aim of this paper is to extend the theory developed in the previous references by using IQCs as exposed in \cite{scherer2018stability}. In this context, a filter providing the projection coefficients of the infinite-dimensional state is derived. Thanks to this filter, a new class of IQCs can be generated, not related to a specific class of PDEs. In that sense, the proposed result is adaptable to many linear PDEs and it is promising to extend these ideas to other classes.


In Section~2, the problem is stated and a stability theorem is derived. The filter 
is obtained in Section~3. Section~4 is dedicated to the application of the previous results to a TDS.
Section~5 provides a numerical example with a transport equation.
Finally, conclusions are drawn and several perspectives are proposed.

\textbf{Notation:} $\mathbb{R}^{n \times m}$ stands for the set of real matrices with $n$ rows and $m$ columns and $\col(A, B) = \left[ A^{\top} B^{\top} \right]^{\top}$ for real matrices $A$ and $B$ of appropriate dimensions. $\mathbb{S}^n$ is the set of symmetric matrices
A matrix $A$ is positive definite if $A \succ 0$ 
and positive semi-definite if $A \succeq 0$. For a matrix $A$, $A^{\perp}$ is a basis of its null-space and for square matrices, we use $\diag(A, B) = \left[ \begin{smallmatrix} A & 0 \\ 0 & B \end{smallmatrix} \right]$ and $\He(A) = A + A^{\top}$. The Euclidean norm of $v$ is defined as $\| v \|^2 = v^{\top} v$. For brevity, the following notation is used for $\A, \B, \C, \D, P$ and $M$ matrices of appropriate dimensions:
\[
	\mathscr{L} \left(P, M, \left[ \begin{smallmatrix} \A & \B \\ \C & \D \end{smallmatrix} \right] \right) =
	 \left[ \begin{smallmatrix} I & 0 \\ \A & \B \\ \C & \D \end{smallmatrix} \right]^{\top} \left[ \begin{smallmatrix} 0 & P & 0 \\ P & 0 & 0 \\ 0 & 0 & M \end{smallmatrix} \right] \left[ \begin{smallmatrix} I & 0 \\ \A & \B \\ \C & \D \end{smallmatrix} \right].
\]
We also use $C (sI - A)^{-1} B + D =$ \scalebox{0.8}{$\left\llbracket \begin{array}{c|c} A & B \\ \hline C & D \end{array} \right\rrbracket$}.\\
$\partial_x^{(i)}$ means the $i^{\text{th}}$ distributional derivative with respect to the variable $x$. For convenience, a short-hand notation is $\partial_x f = f_x$. The space of square integrable functions on $[0, 1]$ is denoted by $L_2(0,1)$ and $L_{2e}$ is the space of locally square integrable signals on $[0, \infty)$. We work with Sobolev spaces denoted by
\[
	H^m(0,1) = \left\{ f \in L_2(0,1) \ | \ \forall k \leq m, \partial_x^{(k)} f \in L_2(0,1) \right\}.
\]
The canonical inner product on $L_2(0,1)$ is $\langle f, g \rangle = \int_0^1 f^{\top}(s) g(s) ds$ for $f,g \in L_2(0,1)$
. 
Similarly, $L_2(0, \infty)$ is the space of square integrable functions on $[0, \infty)$. The binomial coefficients are defined by $\left( \begin{smallmatrix} k \\ l \end{smallmatrix} \right) = \frac{k!}{l! (k-l)!}$. 

\section{Problem Statement}

\begin{figure}
	\centering
	\includegraphics[width=6cm]{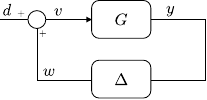}
	\caption{Block diagram describing the interconnected system \eqref{eq:problem}-\eqref{eq:uncertainty}.}
	\label{fig:problem}
\end{figure}
The studied system is presented in Fig.~\ref{fig:problem} and described by the following equations as inspired by the classical IQC framework \citep{megretski1997system}:
\begin{equation} \label{eq:problem}
	\left\{
	\begin{array}{l}
		v = w + d,\ \ (y,w) \in \Delta, \quad \quad  \quad \quad \quad \quad \quad \quad \quad \\
		y = G v.
	\end{array}
	\right.
\end{equation}
Here, $\Delta$ is a relation defined afterwards and $G$ is the linear time-invariant system whose representation is
\begin{equation} \label{eq:ODE}
	\left\{ \begin{array}{l}
		\dot{X} = A X + B v, \quad \quad X(0) = 0_{n,1}, \\
		y = C X + D v,
	\end{array} \right.
\end{equation}
with state, input and output dimensions $n$, $r$ and $p$ respectively; $A$, $B$, $C$ and $D$ are real matrices of appropriate dimensions. 

\begin{remark} Note that we study the interconnection between an LTI system and a relation, similarly to what is done in the behavioral approach to system theory \citep{willems2013introduction}. We use a relation since, in our situation, there is no clear way to express $\Delta$ as a causal input-output map.
This is different from the feedback configuration as studied by Megretski and al. (\citeyear{megretski1997system}). \end{remark}

Before introducing $\Delta$, let us define the following infinite-dimensional system
with state $z(x,t) \in \mathbb{R}^l$ for $t \geq 0$:
\begin{equation} \label{eq:uncertainty}
	 z_t(x,t) = \operator z(x,t) = \sum_{k=0}^{m} F_k \partial_{x}^{(k)} z(x,t), \text{ for } x \in [0, 1].
\end{equation}
The initial state is $z(\cdot, 0) = 0_{l,1}$. Here $m \in \mathbb{N}$, $m > 0$, $F_{m} \neq 0$ and the domain of the operator $\operator$ is parameterized by $y \in \R^p$ such that
\begin{equation} \label{eq:domainParam}
	\domain_{y}(\operator) = \left\{ z \in H^m(0, 1)\ |\  K_{\boundary} \boundary_{m}(z) = K y\right\}
\end{equation}
where $K_{\boundary} \in \mathbb{R}^{m l \times 2 m l}$ has full-rank, $K \in \mathbb{R}^{m l \times p}$, and
\[
	\boundary_{m}(z) = \col \left( z(0), \ z(1), \dots, \partial_x^{(m-1)}z(0), \ \partial_x^{(m-1)} z(1) \right).
\]
Note that $\boundary_{m}$ is similar to the trace operator
and maps the input function to its values and those of the derivatives at the two boundary points. Such an operator has already been introduced by \cite{AHMADI2016163} for instance. The equality $K_{\boundary} \boundary_{m}(z) = K y$ reflects the boundary conditions.

$\Delta$ is the relation that is implicitly defined through the boundary conditions of \eqref{eq:uncertainty}. We say that all pairs $(y,w) \in L_{2e}^2$ belong to $\Delta$ if there exists a solution $z \in C^1([0, \infty], H^m(0,1))$ of \eqref{eq:uncertainty} with
\begin{equation} \label{eq:implicit}
	\quad z(.,t)\in\domain_{y(t)}(\operator),\quad w(t)=L\boundary_{m}(z(.,t))
\end{equation}
for all $t \geq 0$, where $L \in \mathbb{R}^{r \times 2 m l}$.

\begin{remark} Wellposedness is assumed in this paper, as it is usually the case in IQC theory (Megretski et al., \citeyear{megretski1997system}) and examples show that $\Delta$ is often not empty. Theorem~3.1.7 in combination with Theorem~3.3.4 from \cite{curtain1995introduction} with $KD = 0$ and $d \in L_{2e}$ can be used to obtain general sufficient conditions for $\Delta$ to be non-empty. \end{remark}

To see that the proposed framework is versatile, the following examples describe how a time-delay system and a coupled ODE/heat equation can be transformed into \eqref{eq:problem}.

\begin{example} \label{ex:step1} We consider the time-delay system
\[
	\left\{ \begin{array}{ll}
		\dot{X}(t) = AX(t) + B (X(t-h) + d(t)), \quad & t > 0, \\
		y(t) = CX(t) + DX(t-h), & t \geq 0, \\
	\end{array} \right.
\]
with $A \in \mathbb{R}^{n \times n}, \ B \in \mathbb{R}^{n \times n}, C = I_n, D = 0_{n, n}$ and $h > 0$. This system has been widely studied by \cite{norm,5687820} for instance.

This system can be converted into an ODE that is coupled to a transport equation with speed $\rho = h^{-1}$. The PDE is written as
\begin{equation} \label{eq:transport}
	z_t (x,t) = - \rho z_x (x,t), \quad x \in (0, 1), \quad t > 0
\end{equation}
with $w(t) = z(1,t) = L \boundary_1(z(t))$ and 
\[
	\domain(\operator) = \left\{ z \in H^1(0, 1) \ | \ z(0,t) = y(t) \right\}.
\]
In our framework we have $m = 1, F_0 = 0, F_1 = -\rho I_n$ with $\rho = h^{-1} > 0$ and $l = n$. The boundary conditions are described with $K = I_n$ and $K_{\boundary} = \left[ \begin{matrix} I_n & 0_n \end{matrix} \right]$, while the signal $w$ is expressed with $L = \left[ \begin{matrix} 0_n & I_n \end{matrix} \right]$.
 \end{example}
 
 \begin{example} \label{ex:step1Heat} Consider now the interconnection studied in \cite{BAUDOUIN2019195} between \eqref{eq:ODE} and a heat equation
 \[
 	z_t(x, t) = \gamma z_{xx}(x,t) 
 \]
 with $z(0,t) = C X(t)$, $u_x(1,t) = 0$ for $x \in (0, 1)$, $t \geq 0$. It can be written in the proposed framework by considering $m = 2$, $F_0 = 0$, $F_1 = 0$, $F_2 = \gamma$. The boundary conditions are such that $K_{\boundary} = \left[ \begin{smallmatrix} I & 0 & 0 & 0 \\ 0 & 0 & 0 & I \end{smallmatrix} \right]$ and $K = \left[ \begin{smallmatrix} I & 0 \end{smallmatrix} \right]^{\top}$. The signal $w$ is generated considering $L = [ 0 \ I ]$.
 \end{example}

\begin{remark} Note that the system~\eqref{eq:uncertainty} is a generic PDE. 
Other physical models fitting the framework 
are, for instance, the linear KdV equation \citep{miles1981korteweg} with $m=3$ or the beam equation with $m=4$ \citep{gupta1988existence}.
\end{remark}

We will study the stability of \eqref{eq:problem} according to the following definition, which is classical in IQC theory. 
\begin{de}
System~\eqref{eq:problem} is \textbf{stable} if there exists $\gamma$ such that the following inequality holds for all its trajectories:
\begin{equation} \label{eq:stable}
	\forall T \geq 0, \quad \int_0^T \!\! \| y(s) \|_{\R^n}^2 ds \leq \gamma^2 \int_0^T \!\! \| d(s) \|_{\R^p}^2 ds.
\end{equation}
\end{de}

Our goal is to obtain computational tests ensuring that system~\eqref{eq:problem} is stable.

\section{Robust IQC Theorem on Signals}

An Integral Quadratic Constraint on $y, \boundary_m(z) \in L_2(0, \infty)$ is expressed as:
\begin{equation} \label{eq:IQC}
	\int_{-\infty}^{+\infty} \begin{bmatrix} \hat{y}(j \omega) \\ \widehat{\boundary_m(z)}(j \omega) \end{bmatrix}^{\top} \Pi(j \omega) \begin{bmatrix} \hat{y}(j \omega) \\ \widehat{\boundary_m(z)}(j \omega) \end{bmatrix} d\omega \geq 0.
\end{equation}
Here, the Fourier transforms are denoted by a $\hat{\cdot}$. $\Pi$ is a frequency-dependent Hermitian matrix of appropriate dimension and called the multiplier in the sequel. We assume that $\Pi$ can be written in the form
\begin{equation} \label{eq:factorization}
	\Pi = \Psi^* M \Psi
\end{equation}
where $\Psi$ is a transfer matrix with poles in the open left-half plane, $\Psi^*(s) = \Psi(-s)^{\top}$ and $M = M^{\top}$. If \eqref{eq:IQC} holds for $\Pi_1 = \Psi^* M_1 \Psi$ and $\Pi_2 = \Psi^* M_2 \Psi$, then it is also valid for $\Pi_1 + \Pi_2 = \Psi^*(M_1 + M_2)\Psi$. This multiplier representation is classical and many classes of multipliers admit such a description \citep{megretski1997system,veenman2016robust}. Using Parseval's theorem and \eqref{eq:factorization}, a time-domain version of \eqref{eq:IQC}, called a soft IQC, with \eqref{eq:factorization} can be deduced:
\begin{equation} \label{eq:softIQC}
	\int_0^{+\infty} \psi(t)^{\top} M \psi(t) dt \geq 0.
\end{equation}
Here $\psi$ is the output of the system defined by $\Psi$ with input $y, \boundary_m(z) \in L_2(0, \infty)$ associated to a state-space realization:
\[
	\Psi = \left\llbracket \begin{array}{c|cc} A_{\Psi} & B_{y} & B_{\boundary} \\ \hline C_{\Psi} & D_{y} & D_{\boundary} \end{array} \right\rrbracket.
\]
Therefore, $\Psi$ has the following time-domain representation
\begin{equation} \label{eq:filter}
	\begin{array}{l}
		\dot{\xi} \hspace{0.09cm}= A_{\Psi} \xi + B_{y} y + B_{\boundary} \boundary_m(z), \quad \quad \xi(0) = 0, \\
		\psi = C_{\Psi} \xi + D_{y} y + D_{\boundary} \boundary_m(z),
	\end{array}
\end{equation}
where $\xi(t) \in \R^{n_{\xi}}$, $\psi(t) \in \mathbb{R}^{p_{\psi}}$ and $A_{\Psi}$, $B_y$, $B_{\boundary}$, $C_{\Psi}$, $D_{y}$ and $D_{\boundary}$ are of appropriate dimensions. 
This filter can be seen as an augmented system providing new signals, which are of interest when finding suitable multipliers $M$ in order to get valid IQCs.

Notice that \eqref{eq:softIQC} is an infinite horizon IQC. Considering only truncations of the signal $\psi$ to $[0, T]$ for $T > 0$, a revised notion, called finite horizon IQC has been proposed in \cite{megretski1997system}. This last notion has been enhanced in \cite{scherer2018stability} introducing a terminal cost depending on the final value of the filter. This notion, useful for our purpose, is recalled here.
\begin{de} For the trajectories of \eqref{eq:uncertainty}, \eqref{eq:domainParam} and \eqref{eq:filter} with any $y, \boundary_m(z) \in L_{2e}$,
$M = M^{\top}$ is a multiplier for \textbf{a finite horizon IQC with terminal cost} defined by the matrix $Z = Z^{\top}$ if the following holds for all $T \geq 0$:
\begin{equation} \label{eq:finiteHorizonIQC}
	\int_0^T \psi^{\top}(t) M \psi(t) dt + \xi^{\top}(T) Z \xi(T) \geq 0.
\end{equation}
We then say that $M$ defines a finite horizon IQC with terminal cost $Z$ with respect to the filter $\Psi$.
\end{de}
Notice that as we are focusing on signals in $L_{2e}$ in the previous definition and not in $L_2$ then the filter $\Psi$ can be unstable. In the sequel, and similarly to what was done by \cite{scherer2018stability}, we introduce the following realization:
\begin{equation} \label{eq:newFilter}
	\Psi \left( \begin{matrix} GL \\ I_{2(m+1)l} \end{matrix} \right) = \left\llbracket \begin{array}{cc|c} A_{\Psi} & B_y C & B_y D L + B_{\boundary} \\ 0 & A & BL \\ \hline C_{\Psi} & D_{y} C & D_y D L + D_{\boundary} \end{array} \right\rrbracket = \left\llbracket \begin{array}{c|c} \A & \B \\ \hline \C & \D \end{array} \right\rrbracket.
\end{equation}


All these definitions lead to the following theorem which is an adapted version of Theorem~3 in Scherer et al. (\citeyear{scherer2018stability}).

\begin{theo} \label{theo:IQC} Let $P \in \mathbb{S}^{n_{\xi} + n}$ and a multiplier $M = M^{\top}$ be such that the finite horizon IQC with terminal cost $Z \in \mathbb{S}^{n_\xi}$ (defined in \eqref{eq:finiteHorizonIQC} with respect to the factorization \eqref{eq:factorization}) holds together with
\begin{equation} \label{eq:kernelKYP}
	{K_{\circ}^{\perp}}^{\top} \L \left( P, M, \left[ \begin{smallmatrix} \A & \B \\ \C & \D \end{smallmatrix} \right] \right) K_{\circ}^{\perp} \prec 0,
\end{equation}
\begin{equation} \label{eq:positivity}
	P - \left[ {\scriptsize \begin{array}{c|c} Z & 0 \\ \hline 0 & 0_n \end{array} } \right] \succ 0,
\end{equation}
where
\[
	K_{\circ} = \left[ \begin{matrix} 0_{ml,n_{\xi}} & K C & - K_{\boundary}
	\end{matrix} \right].
\]
Then system~\eqref{eq:problem} is stable in the sense of Definition~1.
\end{theo}

The main advantage of this theorem lies in the simplicity of its proof.


\begin{proof} To ease the readability, the sketch of proof is divided into two parts. First, we show that if \eqref{eq:kernelKYP} holds, then an extended LMI also holds. The second part aims at deriving \eqref{eq:stable} from this larger LMI; the calculations are highly inspired by the ones in \cite{scherer2018stability}.

\textit{Step 1:} Take first a trajectory $\eta = \col(\xi, X, \boundary_m, d)$ of \eqref{eq:ODE} and \eqref{eq:uncertainty} together with \eqref{eq:implicit}. We can then introduce the dynamical system
\begin{equation} \label{eq:system2}
	\left\{
	\begin{array}{cl}
		\frac{d}{dt} \left[ \begin{smallmatrix} \xi \\ X \end{smallmatrix} \right] \!\!\!\!&= \A \left[ \begin{smallmatrix} \xi \\ X \end{smallmatrix} \right] + \tilde{\B} \left[ \begin{smallmatrix} \boundary_m(z) \\ d \end{smallmatrix} \right], \quad \quad \left[ \begin{smallmatrix} \xi(0) \\ X(0) \end{smallmatrix} \right] = 0, \\
		\left[ \begin{smallmatrix} \psi \\ y \\ d \end{smallmatrix} \right] \!\!\!\!&= \tilde{\C} \left[ \begin{smallmatrix} \xi \\ X \end{smallmatrix} \right] + \tilde{\D} \left[ \begin{smallmatrix} \boundary_m(z) \\ d \end{smallmatrix} \right],
	\end{array}
	\right.
\end{equation}
with
\[
	\begin{array}{c}
		\tilde{\B} = \left[ \begin{smallmatrix} \B & \begin{smallmatrix} B_y D \\ B \end{smallmatrix} \end{smallmatrix} \right], \quad \
		\tilde{\C} = \left[ \begin{smallmatrix} \C \\ \begin{smallmatrix} 0 & C \\ 0 & 0 \end{smallmatrix} \end{smallmatrix} \right], \quad \
		\tilde{\D} = \left[ \begin{smallmatrix}
			 \D & D_y \\
			 D L & D \\
			 0 & I_r
		\end{smallmatrix} \right].
	\end{array}
\]
Note that, for almost all $t \geq 0$, $\tilde{K}_{\circ} \eta(t) = 0$ where 
$\tilde{K}_{\circ} = \left[ \begin{smallmatrix} K_{\circ} & 0 \end{smallmatrix} \right]$. 
Let $\tilde{M}_{\gamma} = \diag(M, \frac{1}{\gamma} I_{p}, -\gamma I_r)$ and define the following matrices:
\[
\L = \L \left( P, M, \left[ \begin{smallmatrix} \A & \B \\ \C & \D \end{smallmatrix} \right] \right) \text{ and } \tilde{\L}_{\gamma} = \L \left( P, \tilde{M}_{\gamma}, \left[ \begin{smallmatrix} \A & \tilde{\B} \\ \tilde{\C} & \tilde{\D} \end{smallmatrix} \right] \right).
\]
After some calculations, the following equality is obtained:
\[
	\tilde{\L}_{\gamma} = \left[ \begin{matrix} \L & \Phi(P, \B, \C, M, B_y, D_y) \\ \star & D_y^{\top} M D_y - \gamma I \end{matrix} \right] + \frac{1}{\gamma} \Theta (C, D, L),
\]
where $\Phi$ and $\Theta$ are matrices of appropriate dimension that are not depending on $\gamma$.
Assume now that the conditions of Theorem~\ref{theo:IQC} hold. Then, by Finsler's lemma \citep{finsler}, there exists $\sigma > 0$ such that:
\[
	\L - \sigma K_{\circ}^{\top} K_{\circ} \prec 0.
\]
Consequently, since $\Phi$ and $\Theta$ do not depend on $\gamma$, there exists $\gamma > 0$ sufficiently large such that we get $\tilde{\L}_{\gamma} - \sigma \tilde{K}_{\circ}^{\top} \tilde{K}_{\circ} \prec 0$. Using Finsler's lemma again yields:
\begin{equation} \label{eq:kernelKYP2}
	\tilde{K}_{\circ}{{}^{\perp}}^{\top} \tilde{\L}_{\gamma} \tilde{K}_{\circ}^{\perp} \prec 0.
\end{equation}

\textit{Step 2:} 
Let us define $\chi(t) = \left( \tilde{K}_{\circ}^{\perp} \right)^+ \eta(t)$ for almost all $t \geq 0$, where $\left( \tilde{K}_{\circ}^{\perp} \right)^+$ is a right-inverse of $\tilde{K}_{\circ}^{\perp}$.
Inequality \eqref{eq:kernelKYP2} can be written as
\begin{equation} \label{eq:L}
	\chi^{\top}(t) \tilde{K}_{\circ}{{}^{\perp}}^{\top} \tilde{\L}_{\gamma} \tilde{K}_{\circ}^{\perp} \chi(t) = \eta^{\top}(t) \tilde{\L}_{\gamma} \eta(t) \leq 0.
\end{equation}
Expanding \eqref{eq:L} and using \eqref{eq:system2} yields for almost all $t \geq 0$:
 \begin{multline*}
 	\frac{d}{dt} \left( \left[ \begin{smallmatrix} \xi(t) \\ X(t) \end{smallmatrix} \right]^{\top} P \left[ \begin{smallmatrix} \xi(t) \\ X(t) \end{smallmatrix} \right] \right) + \psi^{\top}(t) M \psi(t) \\
 	+ \frac{1}{\gamma} \| y(t) \|_{\R^p}^2 - \gamma \| d(t) \|_{\R^p}^2 \leq 0.
 \end{multline*}
An integration between $0$ and $T \geq 0$ leads to
\begin{multline*}
 	\left[ \begin{smallmatrix} \xi(T) \\ X(T) \end{smallmatrix} \right]^{\top} P \left[ \begin{smallmatrix} \xi(T) \\ X(T) \end{smallmatrix} \right] + \int_0^T \psi^{\top}(t) M \psi(t) dt \\
 	+ \int_0^T \frac{1}{\gamma} \| y(t) \|_{\R^p}^2 - \gamma \| d(t) \|_{\R^p}^2 dt \leq 0.
\end{multline*}
Using IQC~\eqref{eq:finiteHorizonIQC}, we get for all $T \geq 0$:
\begin{multline*}
	\left[ \begin{smallmatrix} \xi(T) \\ X(T) \end{smallmatrix} \right]^{\top} P \left[ \begin{smallmatrix} \xi(T) \\ X(T) \end{smallmatrix} \right] - \xi^{\top}(T) Z \xi(T) \\
 	+ \int_0^T \frac{1}{\gamma} \| y(t) \|_{\R^p}^2 - \gamma \| d(t) \|_{\R^p}^2 dt \leq 0.
\end{multline*}
Since \eqref{eq:positivity} holds, we conclude that system~\eqref{eq:problem} is stable.
\end{proof}

\begin{remark} Note that there are some differences between the previous theorem and Theorem~1 derived by Megretski et al. (\citeyear{megretski1997system}). For instance, we neither require the signals to be in $L_2(0, \infty)$ nor use the causality of the relation.\\
The main advantage of the current formulation is that it is more general than in \cite{megretski1997system} and it is possible to recover existing results of the IQC theory, as shown in \cite{scherer2018stability}.
\end{remark}




\section{The Projection Methodology}
\label{sec:filter}

In this paper, we want to design a filter $\Psi$ which is well-suited for the analysis of infinite dimensional uncertainties such as \eqref{eq:uncertainty}.
We will use the projection methodology, an idea firstly introduced for time-delay systems in \cite{seuret:hal-01065142}. It can be extended to coupled ODE/PDE problems as explained below.



Let $N \geq 0$ and denote by $\left(e_k\right)_{k \leq N}$ an orthogonal family of $L_2(0, 1)$. The projection coefficients of $z(\cdot, t)$ are then defined for $k \in  \N$ as
\[
	\Omega_k(t) = \int_0^1 z(x,t) e_k(x) dx.
\]
$\Omega_k$ is called the projection coefficient of order $k$ of $z$. The filter should then generate the projections coefficients of $z(\cdot, t)$ from its boundary values. Notice that the signals $\Omega_k$ can be generated by a marginally stable linear system driven by $\boundary_{m}(z)$. Indeed, for $m = 2$, the differentiation of $\Omega_k$ with respect to time leads to:
\[
	\begin{array}{ll}
		\!\!\dot{\Omega}_k(t) \!&\displaystyle=\!\int_0^1\! z_t(x,t) e_k(x) dx = \sum_{i=0}^2 F_i \! \int_0^1 \!\!\partial_x^{(i)} z(x,t) e_k(x) dx.
	\end{array}
\]
\[
    \begin{array}{ll}
		\!\!\dot{\Omega}_k(t) \!&\displaystyle = F_2 \left( \left[ z_x(x,t) e_k(x) \right]_0^1 - \left[ z(x,t) \frac{d}{dx} e_k(x) \right]_0^1 \right. \\
		&\displaystyle \hfill \left. + \int_0^1 z \frac{d^2}{dx^2}e_k(x) dx \right) \\
		&\displaystyle \hspace*{-0.75cm} + F_1 \left( \left[ z(x,t) e_k(x) \right]_0^1 -\! \int_0^1 \!\! z(x,t) \frac{d~\!e_k}{dx} (x) dx \right) + F_0 \Omega_k(t).
	\end{array}
\]
If we choose $\left\{ e_k \right\}_{k \in \mathbb{N}}$ as a polynomial orthogonal family of $L_2(0,1)$, $\frac{d}{dx}e_k$ is a linear combination of strictly lower order polynomials in the same family. The only polynomial basis of $L_2(0,1)$ which has values at the boundaries are the Legendre polynomials $\{ \mathcal{L}_k \}_{k \in \mathbb{N}}$. This family has the following properties \citep{courant1966courant}:
\[
	\hspace{-0.15cm}
	\begin{array}{l}
		\displaystyle \mathcal{L}_k(x) = (-1)^k \sum_{l = 0}^k (-1)^l \left( \begin{smallmatrix} k \\ l \end{smallmatrix} \right) \left( \begin{smallmatrix} k+l \\ l \end{smallmatrix} \right) x^l, \\
		\displaystyle \mathcal{L}_k(1) = 1, \quad \mathcal{L}_k(0) = (-1)^k, \quad \| \mathcal{L}_k \|_{L_2}^2 = (2k+1)^{-1}, \\
		\displaystyle \frac{d}{dx}\mathcal{L}_k(x) = \!\sum_{j=0}^k\! \ell_{kj} \mathcal{L}_j(x), \ \frac{d^2}{dx^2}\mathcal{L}_k(x) = \!\sum_{j=0}^k \sum_{i=0}^{j}\! \ell_{kj} \ell_{ji} \mathcal{L}_i(x), \\
	\end{array}
\]
where $\ell_{kj} = (2j+1)(1 - (-1)^{k+j})$ if $j \leq k$ and $\ell_{kj} = 0$ otherwise. Note that $\dot{\Omega}_k$ is a linear combination of strictly lower order projection coefficients, which provides a hierarchical structure, expressed explicitly as follows:
\[
	\begin{array}{ll}
		\dot{\Omega}_k(t) \!&= \displaystyle F_2 \left( z_x(1,t) - (-1)^k z_x(0,t) - \sum_{j=0}^k \ell_{kj} z(1,t) \ \ \ \  \right. \\
		& \displaystyle \hfill \left.+ \sum_{j=0}^k (-1)^j \ell_{kj} z(0,t) + \sum_{j=0}^k \sum_{i=0}^{j} \ell_{kj} \ell_{ji} \Omega_i(t) \right) \\
		& \displaystyle \quad + F_1 \! \left( z(1,t) - (-1)^k z(0,t) - \sum_{j=0}^k \ell_{kj} \Omega_j \! \right) \!+\! F_0 \Omega_k.  \\
	\end{array}
\]
The behaviors of the projection coefficients are highly related to the behavior of $z(\cdot, t)$ and can be generated by a linear sub-system whose inputs are the boundary conditions. It is natural to define the state of the filter by stacking the projection coefficients as in
\[
	\xi_N(t) = \col \left( \Omega_0(t), \ \Omega_1(t), \ \cdots, \ \Omega_N(t) \right)\!.
\]
Finally, the filter dynamics can be expressed by \eqref{eq:filter} in the case $m \in \{ 1, 2\}$, $n_{\xi} = (N+1)l$ and $p_{\psi} = 2ml+n_{\xi}$ with
\begin{equation} \label{eq:filter2}
    \hspace*{0.3cm}
	\begin{array}{l}
		\hspace{-0.4cm}\displaystyle A^N_{\Psi} = \sum_{i=0}^{m} (-1)^i \tilde{F}^N_i L_N^i, \quad B^N_{y} = 0_{n_{\xi}, p}, \\ \hspace{-0.4cm}\displaystyle B_{\boundary}^N = \left[ \begin{matrix} B_1 & \cdots & B_{m} \end{matrix} \right], \ B^N_i = \sum_{j=i}^m \tilde{F}_{j} (-L_N)^{j-i} \left[ \begin{matrix} - \mathbb{1}_N^* & \mathbb{1}_N  \end{matrix} \right], \quad \quad \quad \quad \\
		\hspace{-0.4cm}C^N_{\Psi} = \left[ \begin{matrix} 0_{2 (m+1) l, n_{\xi}} \\ I_{n_{\xi}} \end{matrix} \right]\!\!, \quad D_\boundary^N = \left[ \begin{matrix} I_{2(m+1)l} \\ 0_{n_{\xi},2(m+1)l} \end{matrix} \right], \ D^N_{y} = 0_{p_{\psi}, p}, \quad \displaystyle 
	\end{array}
	\vspace*{-0.3cm}
\end{equation}
where
\[
	\begin{array}{ll}
		L_N = [ \ell_{ij} I_l ]_{i,j=0..N} & \in \mathbb{R}^{n_{\xi} \times n_{\xi}}, \\
		\tilde{F}^N_i = \diag(F_i, \cdots, F_i) &  \in \mathbb{R}^{n_{\xi} \times n_{\xi}}, \\
		\mathbb{1}_N = \left[ \begin{matrix} I_l & I_l & \cdots & I_l \end{matrix} \right]^{\top}\!\!\!\!\! & \in \mathbb{R}^{n_{\xi} \times l}, \\
		\mathbb{1}_N^* = \left[ \begin{matrix} I_l & -I_l & \cdots & (-I_l)^{N} \end{matrix} \right]^{\top}\!\!\! \ & \in \mathbb{R}^{n_{\xi} \times l} .
	\end{array}
\]
Notice that the previous equation can be extended to higher values of $m$ with a proof based on induction.

To summarize, the filter $\Psi_N$ has the following form:
\begin{equation} \label{eq:filterN}
	\Psi_N : \begin{array}[t]{ccc}
		L_{2e}(0, \infty)^{p + 2 m l} & \to & L_{2e}(0, \infty)^{ p_{\psi}} \\
		\left[ \begin{smallmatrix} y \\ \boundary_{m}(z) \end{smallmatrix} \right] & \mapsto & \psi_N =  \left[ \begin{smallmatrix} \boundary_{m}(z) \\ \xi_N \end{smallmatrix} \right],
	\end{array}		
\end{equation}
and from the boundary values of $z$ and its derivatives, it computes the $N^{\text{th}}$ projection coefficients of $z$. Thanks to the previous calculations, using this filter in Theorem~\ref{theo:IQC} leads to the following definitions:
\[
	\begin{array}{llll}
		\!\A = \left[ \begin{matrix} A^N_{\Psi} & 0_{} \\ 0 & A \end{matrix} \right]\!, \ \ \B = \left[ \begin{matrix} B_{\boundary}^N \\ BL \end{matrix} \right]\!, \ \
		\C = \left[ \begin{matrix} C^N_{\Psi} & 0_{} \end{matrix} \right]\!, \ \ \D = D_{\boundary}^N.
	\end{array}
\]


Note that the filter $\Psi_N$ is depending on the dynamic of the uncertainty. Finding the multiplier $M_N$ is then a consequence of the order $N$ used for the filter and the PDE under consideration. 
To apply Theorem~\ref{theo:IQC}, we can follow these subsequent steps:
\begin{enumerate}
	\item First rewrite the system under consideration to fit equations~\eqref{eq:ODE} and \eqref{eq:uncertainty} (see Examples~\ref{ex:step1} and \ref{ex:step1Heat});
	\item Find a class of multiplier $M_N$ and terminal cost matrices $Z_N$ verifying \eqref{eq:finiteHorizonIQC};
	\item Test the feasibility of LMIs \eqref{eq:kernelKYP} and \eqref{eq:positivity} using the multipliers found in the previous step. If they are feasible, then the system is stable (see Section~6).
\end{enumerate}

The following section is dedicated to an example of step~2 with $m = 1$.

\section{Generating IQCs using Projections}

Proposing IQCs for infinite-dimensional systems requires different tools than in the finite-dimensional case. We introduce in the first subsection a justification of how the filter designed in the previous subsection can be helpful. Then, the second subsection is dedicated to an application to a transport equation as defined in Example~1.

\subsection{General idea}

Since $z(\cdot, t) \in H^{m}([0, 1])$, it is difficult to find dissipation inequalities. One solution is to consider an approximation of $z$ on a finite-dimensional space. Let $N \geq 0$ and denote by $\left(e_k\right)_{k \leq N}$ an orthogonal family of $L_2(0, 1)$. Then, the following holds for $t \geq 0$:
\[
	\begin{array}{rl}
		\displaystyle \min_{y \in \text{Span}(e_k)_{k \leq N}} \!\!\!\!\!\!\! \| z(\cdot,t) - y \|_{L_2}^2 &= \| z(\cdot, t) - z_N(\cdot, t) \|_{L_2}^2 \\
		&\hspace{-1.55cm}= \|z(\cdot,t)\|_{L_2}^2 - 2 \langle z(\cdot, t), z_N(\cdot, t) \rangle + \|z\|_{L_2}^2\\
		&\hspace{-1.55cm}= \| z(\cdot, t) \|_{L_2}^2 - \|z_N(\cdot, t) \|_{L_2}^2 \geq 0.
	\end{array}
\]
Here we have
\[
	z_N(\cdot,t) = \sum_{k=0}^N \Omega_k(t) \frac{e_k(\cdot)}{\| e_k \|_{L_2}^2}.
\]
The previous equality shows that $z_N$ is the projection of $z$ on the subspace spanned by the family $(e_k)_{k \leq N}$ and is consequently the optimal approximation (with respect to the norm $\| \cdot \|_{L_2}$) of $z$ in the former family. Note that for $R \in \mathbb{S}^l_+$, a consequence of the previous statement is the following:
\begin{equation} \label{eq:bessel}
	\int_0^1 z^{\top}(x, t) R z(x,t) dx \geq \sum_{k=0}^N \frac{1}{\| e_k \|_{L_2}^2 } \Omega_k^{\top}(t) R \Omega_k(t).
\end{equation}
Inequality~\eqref{eq:bessel} is referred to as Bessel Inequality and, as a consequence of Parseval's theorem, equality holds when $N$ goes to infinity. The previous inequality is the starting point for building multipliers in a similar way as with Jensen's inequality for time-delay systems \citep{norm}. These considerations indicates that $\Omega_k$ are signals of interest to characterize the uncertainty $\Delta$ with the help of the IQC defined in \eqref{eq:finiteHorizonIQC}, since the Bessel inequality relates the energies of the PDE and of a finite-dimensional system.

%
%

\subsection{Application to uncertainties generated by a transport equation}

Example~\ref{ex:step1} shows how to model a time-delay system using our framework. This section is dedicated to finding multipliers for the uncertainty~\eqref{eq:transport} such that inequality~\eqref{eq:finiteHorizonIQC} holds. Two multipliers $M^1_N$ and $M^2_N$ are obtained using different techniques and an heavy use of \eqref{eq:bessel}.

\subsubsection{Finite-horizon IQC with terminal cost.}

Since the transport equation is an hyperbolic equation, it is a lossless system from an energy perspective. This first IQC expresses this physical fact.

\begin{lemma} For the transport equation in \eqref{eq:transport}, if $S \succ 0$, then the finite-horizon IQC \eqref{eq:finiteHorizonIQC} with terminal cost $Z_N^1$ and multiplier $M_N^1$ holds with:
\begin{equation} \label{eq:multiplierTransport1}
	M_N^1 = \!\rho \left[ {\scriptsize \begin{array}{cc|c} S & 0 & 0 \\ 0 & -S & 0 \\ \hline 0 & 0 & 0_{n_{\xi}}\end{array}} \right]\!\!, Z_N^1 = - \diag \left( S, 3S, \dots, (2N+1)S \right).
\end{equation}
\end{lemma}

\begin{proof}
Let us introduce the following notation for $t \geq 0$:
\[
	\mathcal{I}_1(t) = \int_0^1 z^{\top}(x,t) S z(x,t) dx = \|z(\cdot, t)\|_S^2.
\]
$\mathcal{I}_1(t)$ is the energy of the PDE at a given time $t$.
Since $\mathcal{I}_1(0) = 0$, the following holds for $T \geq 0$:
\[
	\begin{array}{rl}
		\mathcal{I}_1(T) &\displaystyle= \int_0^1 \int_0^T \partial_t \left( z^{\top}(x,t) S z(x,t) \right) dt dx \\
		&\displaystyle= 2 \int_0^1 \int_0^T z_t^{\top}(x,t) S z(x,t) dt dx\\
		&\displaystyle= - \rho \! \int_0^T \int_0^1 2 z_x^{\top}(x,t) S z(x,t) dx dt\\
		&\displaystyle=\rho \! \int_0^T \!\!\! z^{\top}\!(0,t) S z(0,t) -\! z^{\top}\!(1,t) S z(1,t) dt.
	\end{array}
\]
From \eqref{eq:bessel}, we get $\mathcal{I}_1(T) \geq \sum_{k=0}^N (2k+1) \Omega_k^ {\top}(T) S \Omega_k(T)$. Combining with the two previous results yields
\begin{multline*}
	\rho \! \int_0^T \!\!\! z^{\top}\!(0,t) S z(0,t) -\! z^{\top}\!(1,t) S z(1,t) dt \\
	- \sum_{k=0}^N (2k+1) \Omega_k^ {\top}(T) S \Omega_k(T) \geq 0.
\end{multline*}
Then, the following IQC with terminal cost $Z$ is obtained:
\begin{equation*} \label{eq:IQCtransport1}
	\int_0^T \psi_N(t)^{\top} M_N^1 \psi_N(t) dt + \xi(T)^{\top} Z_N^1 \xi(T) \geq 0
\end{equation*}
for all $T \geq 0$.
\end{proof}
\begin{remark} Losslessness does not translate into an equality since we used Bessel inequality. Equality is recovered for $N \to \infty$.
\end{remark}

\subsubsection{Finite horizon IQC with $Z = 0$.}

The previous inequality is about the energy balance between the input and the output. In this part, we are interested in the transmission of energy between the input and the PDE itself.

\begin{lemma} For $R \succ 0$, The finite horizon IQC \eqref{eq:finiteHorizonIQC} with terminal cost $Z = 0$ and with multiplier
\begin{equation} \label{eq:multiplierTransport2}
	M_N^2 = \left[ {\scriptsize \begin{array}{cc|c} R & 0 & 0 \\ 0 & 0_{l} & 0 \\ \hline 0 & 0 & - \diag \left( R, 3 R, \dots, (2N+1) R \right) \end{array}} \right]
\end{equation}
 holds for system~\eqref{eq:transport} together with the filter \eqref{eq:filter2}.
\end{lemma}
\begin{proof}
We start this time with an integration by parts:
\begin{equation*}
	\begin{array}{l}
		\displaystyle\int_0^1 \He \left( (1-x) z^{\top}\!(x,t) R z_x(x,t) \right)\!dx = \quad \quad \quad \quad \quad \quad \quad \quad\\
		\displaystyle\hfill 2 \left[ (1-x) z^{\top}\!(x,t) R z(x,t) \right]_0^1 + 2 \| z(\cdot, t) \|_R^2 \quad \quad \\
		\displaystyle\hfill - \int_0^1 \He \left( (1-x) z_x^{\top}(x,t) R z(x,t) \right) dx.
	\end{array}
\end{equation*}
In other words, the following holds:
\begin{equation} \label{eq:transportMultiplier2}
	\begin{array}{l}
		z^{\top}\!(0,t) R z(0,t) = \| z(\cdot, t) \|_R^2 \hspace{3.8cm}\\
		\displaystyle\hfill - \int_0^1 \He \left( (1-x) z^{\top}\!(x,t) R z_x(x,t) \right)\!dx.
	\end{array}
\end{equation}
In a similar way as previously, define the following for $t \geq 0$:
\[
	\mathcal{I}_2(t) = \int_0^1 (1-x) z^{\top}(x,t) R z(x,t) dx \leq \|z(\cdot,t)\|_R^2.
\]
Then, we get for $T \geq 0$:
\[
	\begin{array}{rl}
		\!\mathcal{I}_2(T) &\displaystyle= 2 \int_0^1 (1-x) \int_0^T z^{\top}(x,t) R z_t(x,t) dt dx \\
		&\displaystyle= - 2 \rho \int_0^T \int_0^1 (1-x) z^{\top}(x,t) R z_x(x,t) dx dt.
	\end{array}
\]
Using equation \eqref{eq:transportMultiplier2} and inequality \eqref{eq:bessel}, we get:
\[
	\begin{array}{rl}
		\mathcal{I}_2(T) &\displaystyle= \rho \int_0^T z^{\top}(0,t) R z(0,t) - \|z(\cdot,t)\|_R^2 dt\\
		&\displaystyle \leq \rho \int_0^T z^{\top}(0,t) R z(0,t) \\
		&\displaystyle \hfill - \sum_{k=0}^N (2k+1) \Omega_k^{\top}(t) R \Omega_k(t) dt.
	\end{array}
\]
Since $\mathcal{I}_2(t) \geq 0$ for $t \geq 0$, the previous results lead to:
\[
	\int_0^T \psi_N(t) M_N^2 \psi_N(t) dt \geq 0.
\]
\vspace*{-1.1cm}

\end{proof}

Finally, taking $M_N = M^1_N + M^2_N$ and $Z_N = Z^1_N$, we get that $M_N$ defines an IQC with terminal cost $Z_N$ for the uncertainty described in Example~1. The second step of the methodology has then been successfully applied. The following section is dedicated to an application of Theorem~\ref{theo:IQC} for a time-delay system. Before that, a short discussion on the relation with Lyapunov functionals is included.

\subsubsection{On the relation with Lyapunov functionals.}

A Lyapunov-based stability theorem for this problem has been derived in \cite{safi2017tractable}. The functional used is made up of $\mathcal{I}_1$ and $\mathcal{I}_2$. After some computations, one can show that the LMIs resulting from Theorem~\ref{theo:IQC} are the same than the ones obtained in \cite{safi2017tractable}.

Nevertheless, we provide here a more generic framework and the proposed methodology can benefit from all the advances made with IQCs such as performance analysis and already existing multipliers for some classes of non-linearities (see \citep{veenman2016robust} and references therein). In this way, we can get stronger IQCs, and benefit from all the advances made in that field.

\section{Numerical Example and Discussion}

We consider the time-delay system in Example~1 with:
\[
	A = \left[ \begin{smallmatrix} 0 & 0 & 1 & 0 \\ 0 & 0 & 0 & 1 \\ -10-k & 10 & 0 & 0 \\ 5 & -15 & 0 & -0.25 \end{smallmatrix} \right], \quad \quad B = \left[ \begin{smallmatrix} 0 & 0 & 0 & 0 \\ 0 & 0 & 0 & 0 \\ k & 0 & 0 & 0 \\ 0 & 0 & 0 & 0 \end{smallmatrix} \right],
\]
for $k > 0$. This example describes a regenerative chatter taken from \cite{opac-b1100602,seuret:hal-01065142}. This system is a good toy example as finding the values of $(k, \rho)$ making the system stable is not straightforward. Nevertheless, using the Control Toolbox of Matlab\textregistered \ (as in \citep{seuret:hal-01065142} for instance), the stability areas can be exactly obtained.

A first analysis shows that, in the delay independent case, the maximum allowable $k$ is $k_{max} = 0.3$. Using the Small-Gain theorem leads to the stability for $k \leq 0.104$, for [Megretski, \citeyear{megretski1997system}], the maximum $k$ is $0.265$, while using the method described in this paper, we get $0.299$. Even at low order, the proposed methodology is very efficient for dealing with hyperbolic equations.

We are interested now in the numerical analysis in the case of a fixed $k = 2$. The results are displayed in Table~\ref{tab:ex1}. To obtain this table, we used the multiplier obtained in the previous section and we solved LMIs \eqref{eq:kernelKYP} and \eqref{eq:positivity} with a grid on $h$.

First of all, we can note the hierarchy property of Theorem~\ref{theo:IQC} with this example. Indeed, the stable intervals for $N = 0$ are included in the others obtained for higher values of $N$. And this is also the case for the order $2$ and $5$. We can note that for $N = 5$ or $7$, there are at least two stability pockets detected, meaning that Theorem~\ref{theo:IQC} can recover stable intervals of the form $\rho^{-1} \in [h^-, h^+]$ with a possible non zero $h^-$. And for $N = 7$, the detected stable intervals are close to the exact ones, meaning that the theorem provides a precise estimate of the stability area.

The comparison with other IQC theorems shows that Theorem~\ref{theo:IQC} is different in nature. Indeed, the IQC theorems in \cite{megretski1997system} and \cite{veenman2016robust} provide stability for a constant delay between $0$ and a maximum value. Consequently, they cannot detect stability pockets but they ensure the robustness for the delay varying between the two obtained bounds. 

\begin{table}
	\centering
	\begin{tabular}{c||c|c|c}
		Stable Intervals & $[0, 0.859]$ & $[1.117, 1.264]$ & $[2.75, 3.5]$ \\
		\hline
		Theo. 2, N = 0 & $[0, 0.062]$ & $-$ & $-$ \\
		Theo. 2, N = 2 & $[0, 0.854]$ & $-$ & $-$ \\
		Theo. 2, N = 5 & $[0, 0.859]$ & $[1.123, 1.264]$ & $-$ \\
		Theo. 2, N = 7 & $[0, 0.859]$ & $[1.117, 1.264]$ & $[2.83, 3.36]$ \\
		\hline
		[Megretski, \citeyear{megretski1997system}] & $[0, 0.062]$ & $-$ & $-$ \\
		\ \![Veenman, \citeyear{veenman2016robust}] & $[0, 0.060]$ & $-$ & $-$
	\end{tabular}
	\caption{Table of stable values for $\rho^{-1}$. The first row depicts the exact stable intervals and rows 2 to 5 are coming from Theorem~\ref{theo:IQC}.}
	\label{tab:ex1}
\end{table}

\section{Conclusion \& Perspectives}

This paper presented a preliminary work on the robust analysis of infinite-dimensional systems coupled to an ODE. It introduces a different way to assess its stability using IQCs. This approach gives insights about how to build a filter deriving the projections coefficients which helps characterizing the uncertainty. 
We use a new kind of IQCs, which is valid for many more PDEs and, in that sense, we obtain a robustness result. Moreover, it has been shown with an example that, for the transport equation, it leads to the same result as using one of the most recent Lyapunov functionals in the literature.

Furthermore, this paper offers many perspectives, that were not raised with the Lyapunov approach. For instance, further work concerns the wellposedness of the interconnection and a methodology for deriving the IQCs. New tools were introduced and it would be interesting to pursue this direction in order to give a better interpretation of the filter. The interpretation of each IQC in terms of energy would help building new dissipation inequalities that are less conservative. 

\bibliography{report_draft}

\end{document}